\documentclass[bj, numbers, noshowframe]{imsart}

\usepackage{amsmath, amsfonts, amssymb, amsthm, graphicx}
\usepackage[T1]{fontenc}

\startlocaldefs
\theoremstyle{plain}
\newtheorem{theorem}{Theorem}
\newtheorem{proposition}[theorem]{Proposition}
\newtheorem{corollary}[theorem]{Corollary}

\theoremstyle{definition}
\newtheorem{assumption}[theorem]{Assumption}

\newcommand{\norm}[1]{\left\lVert#1\right\rVert}
\newcommand{\ip}[1]{\left\langle #1 \right\rangle}
\newcommand{\e}{\epsilon}

\newcommand{\tr}{\text{tr}}

\newcommand{\R}{\mathbb{R}}
\newcommand{\Z}{\mathbb{Z}}

\newcommand{\E}{\mathbb{E}}
\newcommand{\Prob}{\mathbb{P}}
\renewcommand{\P}{\mathcal{P}}
\newcommand{\Var}{\text{Var}}

\newcommand{\W}{\mathcal{W}}

\renewcommand{\H}{\mathcal{H}}

\newcommand{\X}{\textbf{X}}

\newcommand{\C}{\mathcal{C}}

\let\temp\phi
\let\phi\varphi
\let\varphi\temp

\newcommand{\eq}[1]{\begin{align*}#1\end{align*}}

\endlocaldefs

\begin{document}

\begin{frontmatter}

\title{Implications of weak convergence rates of Markov transition kernels}
\runtitle{Implications of weak convergence rates of Markov transition kernels}

\begin{aug}

\author[A]{\fnms{Austin}~\snm{Brown}\ead[label=e1]{austinbrown@tamu.edu}}
\address[A]{Department of Statistics, Texas A\&M University, TX, USA \printead[presep={,\ }]{e1}}
\end{aug}

\begin{abstract}
This article extends weak convergence bounds of Markov transition kernels to convergence bounds on the variance of the Markov kernel applied to Lipschitz functions.
In the reversible case, weak convergence rates of the transition kernels imply chi-squared divergence convergence bounds if the density of the initialization measure is Lipschitz. These results provide new tools to establish central limit theorems for Lipschitz functions used in Markov chain Monte Carlo simulations. Applications are explored to the stability of Metropolis-Hastings algorithms in high dimensions, stochastic gradient descent, and solutions to stochastic delay equations.
\end{abstract}


\end{frontmatter}

\section{Introduction}

Obtaining informative convergence bounds of a Markov process to its equilibrium measure is increasingly challenging in high dimensions. 
Popular techniques based on incorporating a local Doeblin condition \cite{MT2009, QinHobert2021} can degenerate in high dimensions or fail entirely.
Convergence of the transition kernels in total variation extends to powerful results in the stronger chi-squared divergence and also central limit theorems for averages of a wide class of unbounded functions \cite{kipnis1986central, dedecker:rio:2000, bakry_rate_2008}.
Alternative techniques based on weak convergence bounds can be simpler to establish and scale to high dimensions avoiding such difficulties \cite{hairer_spectral_2014, qin2021wasserstein, Madras2010}.
However, weak convergence of Markov transition kernels only provides certain guarantees for bounded Lipschitz functions and extensions are less explored.

The main interest of this work is understanding the implications of weak convergence rates and how they can provide stronger guarantees for Markov processes.
%In particular, previous research extends convergence in total variation to obtain central limit theorems and other useful results \cite{roberts_geometric_1997, bakry_rate_2008} but implications of weak convergence seems less understood.
The contributions of this article extend weak convergence bounds of Markov transition kernels to explicit convergence bounds on the variance of the transition kernel applied to unbounded Lipschitz functions. Additionally, the weak convergence is characterized through an equivalence relation. In the reversible case, these results imply explicit convergence bounds in the chi-squared divergence if the density of the initial measure is Lipschitz continuous.
This has potential for novel mixing time analysis of Markov processes \cite{lovasz_random_1993, dwivedi_log-concave_2019} with a wide-class of initialization measures.
These results lead to new technical tools to verify central limit theorems for averages of Lipschitz functions for Markov processes and upper bounds on the variance in the Gaussian limit.

Convergence bounds in this work rely on an arbitrary convergence rate, geometric or subgeometric, and do not require reversibility of the Markov process. 
In the reversible and geometrically converging case, stronger conclusions have been developed \cite{hairer_spectral_2014, brown:jones:2023}.
Many of the results are motivated for subgeometric decay towards the invariant measure that is often empirically observed in simulations.
Implications of weak convergence of Markov transition kernels have been studied asymptotically under additional regularity on the transition kernels \cite{tweedie_modes_1977}.
In general, geometric convergence has been the primary focus throughout the literature but a tail mismatch between the distribution generated from a Markov process and the invariant target measure leads to subgeometric rates of convergence \cite{butkovsky_subgeometric_2014, durmus:etal:2016}.
Sufficient conditions for central limit theorems for unbounded Lipschitz functions have been previously studied under convergence in stronger Wasserstein distances \cite{komorowski_central_2012, jin2020wasserstein}.

Applications are explored to understanding the equilibrium of high-dimensional Metropolis-Hastings algorithms, stochastic gradient descent, and solutions to stochastic delay equations.
We use weak convergence bounds to investigate chi-squared convergence bounds on a Preconditioned Crank-Nicolson (pCN) Metropolis-Hastings process where techniques based on total variation for popular algorithms such as the random-walk Metropolis degenerate on infinite-dimensional spaces.
Metropolis-Hastings in infinite dimensions is a useful model to understanding the scaling of high dimensional spaces often of importance in AI.
In certain applications, techniques to obtain convergence bounds relying on total variation are challenging such as analyzing stochastic gradient descent (SGD) where analysis in weak convergence can lead to informative convergence bounds.
We also study convergence bounds on the variance for the solutions of stochastic delay equations where the wealth of tools for analyzing convergence to equilibrium with drift and minorization conditions are unavailable \cite{MT2009}.
Stochastic delay equations provide an important example with applications to understanding the dynamics of neural networks, stochastic volatility models in finance, and disease transmission in epidemiology.

This article is organized as follows.
Section~\ref{section:prelim} defines preliminary background on Markov processes.
Section~\ref{section:convergence} develops the main contributions and its applications to control the variance of empirical averages to develop central limit theorems.
Section~\ref{section:applications} investigates applications to Metropolis-Hastings algorithms, stochastic gradient descent, and the solutions for stochastic delay equations.

\section{Preliminary background on Markov processes}
\label{section:prelim}

We will assume every set and function are Borel measurable, and every measure is defined with Borel sets unless stated otherwise.
Let $\X$ be a nonempty metric space with metric $d(\cdot, \cdot) : \X \times \X \to \R_+$. 
Let $\Pi$ be a target probability measure on $\X$ and for $p \ge 1$, define $L_p(\Pi)$ as the Lebesgue space with its norm $\norm{\cdot}_{L_p(\Pi)}$.

Let $T$ define the time index set being continuous time $T = [0, \infty)$ or discrete time $T = \Z_+$.
Let $(\P_t)_{t \in T}$ be a Markov semigroup on $L_2(\Pi)$ defined by a Markov process $(X_t)_{t \in T}$ on $\X$ and its Markov transition kernels
\eq{
\P_t(x, \cdot)
= \Prob(X_t \in \cdot \mid X_0 = x)
}
for $x \in \X$ and $t \in T$.
We will use the notation
\eq{
&\P_t f( \cdot ) = \int_{\X} f(y) \P_t(\cdot, dy),
&\nu \P_t( \cdot ) = \int_{\X} \P_t(x, \cdot) \nu(dx)
}
for every $t \in T$, every function $f : \X \to \R$, and every probability measure $\nu$ on $\X$.
This defines the marginal  distribution of the Markov process by $\nu \P_t$ when initialized at a probability measure $\nu$. 
We will assume $\Pi$ is the unique invariant probability measure for the Markov process meaning $\Pi \P_t = \Pi$ for every $t \in T$.
We say the transition kernels $(\P_t)_{t \in T}$ are reversible with respect to $\Pi$ if for bounded functions $f, g : \X \to \R$,
\eq{
\int_{\X} \P_t f g d\Pi
= \int_{\X} f \P_t g d\Pi.
}

\section{Weak convergence implications for Markov processes}
\label{section:convergence}

We are interested in convergence bounds on the variance of conditional averages with respect to the transition kernel $\P_t f$ for functions $f : \X \to \R$, that is,
\[
\norm{ \P_t f - \int_{\X} f d\Pi }_{L_2(\Pi)}^2
= \int_{\X} \left( \P_t f - \int_{\X} f d\Pi \right)^2 d\Pi.
\]
Control over this quantity can lead to strong convergence results on the distribution of the Markov process if $\P_t$ is reversible and $f = d\nu / d\Pi$ is the density for an initialization probability measure $\nu$. 
Then this will be the chi-squared divergence
\[
\chi^2(\nu \P_t \mid \Pi) 
= \norm{ \P_t\left( \frac{d\nu}{d\Pi} \right) - 1 }_{L_2(\Pi)}^2.
\] 
Control over the chi-square divergence is often not accomplished directly but instead by using drift and minorization based techniques to obtain convergence in total variation and then extending these results \cite{MT2009}.
The total variation of interest between the Markov transition kernel and the invariant measure controls bounded functions defined for $x \in \X$ by
\[
\norm{\P_t(x, \cdot) - \Pi}_{\text{TV}}
= \sup_{\norm{f}_{\infty} \le 1} \frac{1}{2} \left| \P_t f(x) - \int_{\X} f d\Pi \right|
\]
where $\norm{f}_{\infty} = \sup_{x \in \X} |f|$.
In the case of total variation convergence bounds on the transition kernels based on drift and minorization conditions \cite{MT2009}, bounds on the variance of $\P_t f$ have been studied \cite{bakry_rate_2008}.

Instead of convergence bounds in total variation, the assumption on the transition kernels we make is based on convergence with respect to the Bounded Lipschitz (BL) norm defined for $x \in \X$ by
\[
\norm{\P_t(x, \cdot) - \Pi}_{BL(d)} 
= \sup\left\{ \left| \P_t f(x) - \int_{\X} f d\Pi \right| : \norm{f}_{\text{Lip}(d)} + \norm{f}_{\infty} \le 1 \right\}
\]
where $\norm{f}_{\text{Lip}(d)} = \sup_{x, y \in \X, x \not= y}|f(y) - f(x)| / d(x, y)$.
This is a substantially weaker discrepancy than total variation as the set of test functions are Lipschitz with constant 1 instead of arbitrarily bounded functions.
Let $\wedge, \vee$ denote the min and max respectively.
This is related to the Wasserstein distance using the standard bounded metric $d \wedge 1$
\[
\W_{d \wedge 1} \left( \P_t(x, \cdot), \Pi \right)
= \inf_{\Gamma(\X \times \cdot) = \P_t(x, \cdot), \Gamma(\cdot \times \X) = \Pi} \int_{\X \times \X} \left[ d(x, y) \wedge 1 \right] \Gamma(dx, dy))
\]
where the infimum is taken over all joint probability measures $\Gamma$ satisfying the marginal constraints.

The main assumption we are concerned with is the transition kernel decay towards its invariant measure over all bounded Lipschitz functions.
\begin{assumption}
\label{assumption:convergence}
For the transition kernels $(\P_t)_{t \in T}$, assume there is a function $M : \X \to [1, \infty)$ and a rate function $R : T \to (0, 1]$ strictly monotonically decreasing to $0$ such that 
\[
\norm{\P_t(x, \cdot) - \Pi}_{BL(d)} 
\le M(x) R(t)
\]
holds for every $x \in \X$.
\end{assumption}

In this assumption, the rate function $R(\cdot)$ is strictly decreasing to $0$ defining a rate of decay of the Markov process towards its invariant measure.
Some examples include polynomial rates $R(t) = (1 + t)^{-\kappa}$ or geometric rates $R(t) = \exp(-\kappa t)$ for some $\kappa > 0$.
We describe the cost of an imperfect initialization by a function $M : \X \to [1, \infty)$.
For vectors, let $\norm{\cdot}_p$ for $p \ge 1$ denote the $p$-norm.
The function $M \propto V$ is often defined through a Lyapunov drift function $V$ up to some constant such as $V(x) = 1 + \norm{x}_2^2$.
Conditions based on local coupling conditions and Lyapunov drift conditions to develop subgeometric convergence bounds in Wasserstein distances have been studied extensively \cite{butkovsky_subgeometric_2014, durmus:etal:2016}.
For example, Assumption~\ref{assumption:convergence} can be verified with a subgeometric drift condition \cite{douc_practical_2004} combined with a local Wasserstein coupling condition.
Assumption~\ref{assumption:convergence} and many previous subgeometric convergence bounds only guarantee control over bounded Lipschitz functions and is limiting in applications.

We define an analogy to the signal to noise ratio with the goal to define how difficult a function is to estimate using a Markov process.
The spread to fluctuation ratio of order $p \ge 1$ for a function $f : \X \to \R$ is defined
\[
SFR_p(f) = 
\frac{
\norm{ f - \int_{\X} f d\Pi }_{L_p(\Pi)}
}{
\norm{f}_{\text{Lip}(d)}
}.
\]
The goal is for this definition to be small when the function is difficult to estimate and larger when the function is easier to estimate.
Here the numerator defines the spread or signal of a function $f$ with respect to the target measure.
In particular, functions that are nontrivial only on regions of small probability of the target measure have small signal.
%A function with a large fluctuation represented by the Lipschitz constant and low spread can be difficult for the Markov process to visit during exploration.
The first result extends convergence of the transition kernels for bounded Lipschitz functions to convergence of the variance for unbounded Lipschitz functions. 

\begin{theorem}
\label{theorem:nonreversible_bound}
Assume the transition kernels $(\P_t)_{t \in T}$ satisfy Assumption~\ref{assumption:convergence}.
Let $q \ge 2$ and let $p > 1$.
For every Lipschitz function $f : \X \to \R$, the following hold:
\begin{enumerate}

\item If $t \in T$ is large enough so that 
$
t \ge R^{-1}\left[ 
SFR_{q p}(f)^{\frac{q p}{q-1}}
\right]
$ 
then
\begin{align}
\norm{ \P_t f - \int_{\X}  f d\Pi}_{L_q(\Pi)}^q
&\le  c_1 \left[ \int_{\X} M^{q-1} d\Pi \right]^{1-\frac{1}{q p}} R(t)^{(q-1) \left( 1-\frac{1}{p} \right)} \norm{ f - \int_{\X}  f d\Pi}_{L_{qp}(\Pi)}^q
\end{align}
for a constant $c_1 \le 2^{2q-1} 3$.

\item For all $t \in T$,
\begin{align}
\norm{ \P_t f - \int_{\X}  f d\Pi}_{L_q(\Pi)}^q
&\le C_{q, \Pi} \left[ \int_{\X} M^{q-1} d\Pi \right]^{1-\frac{1}{q p}} R(t)^{(q-1) \left(1-\frac{1}{p} \right)} \norm{ f }_{\text{Lip}(d)}^q
\end{align}
where $C_{q, \Pi} = 2^{-1} c_1 \left[ 1 - 1/q + (1 + 1/q) \norm{d}_{L_{q p}(\Pi \otimes \Pi)}^{q} \right]$.

\end{enumerate}
Additionally, if $(\P_t)_{t \in T}$ are reversible, then $R(\cdot)$ can be replaced with $R(2 \cdot)$.
\end{theorem}

\begin{proof}
Let $\phi : \X \to \R$  be a Lipschitz function with $\int \phi d\Pi = 0$ and $\norm{\phi}_{\text{Lip}(d)} \le 1$. 
The result is trivial if $\int_{\X} M^{q-1} d\Pi = \infty$ or $\norm{ \phi }_{L_{qp}(\Pi)} = \infty$, so we will assume both are finite.
Let $r > 0$ and define $\psi_r : \X \to \R$ by $\psi_r = (-r) \vee ( r \wedge \phi)$. 
In this case, $\norm{\psi_r}_{\infty} \le r$ and $\norm{\psi_r}_{\text{Lip}(d)} \le \norm{\phi}_{\text{Lip}(d)}$.

By convexity and the invariance of $\Pi$, we have the upper bound
\begin{align}
\int_{\X} | \P_t \phi |^q d\Pi
&= \int_{\X} \left| \P_t\phi - \P_t\psi_r + \int_{\X} \psi_r d\Pi + \P_t\psi_r - \int_{\X} \psi_r d\Pi \right|^q d\Pi 
\nonumber
\\
&\le 2^{q-1} \int_{\X} \left| \P_t\phi - \P_t\psi_r  + \int_{\X} \psi_r d\Pi \right|^q d\Pi + 2^{q-1} \int_{\X} \left| \P_t\psi_r - \int_{\X} \psi_r d\Pi \right|^q d\Pi
\nonumber
\\
&\le 2^{q-1} \int_{\X} \left| \phi - \psi_r  + \int_{\X} \psi_r d\Pi \right|^q d\Pi + 2^{q-1} \int_{\X} \left| \P_t\psi_r - \int_{\X} \psi_r d\Pi \right|^q d\Pi.
\label{eq:separate_parts}
\end{align}

To bound the second part in \eqref{eq:separate_parts}, we have since $\norm{\psi_r}_{\text{Lip}(d)} \le \norm{\phi}_{\text{Lip}(d)}$, 
\begin{align}
\int_{\X} \left| \P_t\psi_r - \int \psi_r d\Pi \right|^q d\Pi 
%&= \int_{\X} \left| \P_{t} \psi_r - \int \psi_r d\Pi \right| \left| \P_{t} \psi_r - \int \psi_r d\Pi \right| d\Pi 
%\nonumber
%\\
&\le 2 \norm{\psi_r}_{\infty} \int_{\X} \left| \P_{t} \psi_r - \int \psi_r d\Pi \right|^{q-1} d\Pi 
\nonumber
\\
&\le 2 \norm{\psi_r}_{\infty} \int_{\X} M^{q-1} d\Pi R(t)^{q-1} ( \norm{\psi_r}_{\text{Lip}(d)} +  \norm{\psi_r}_{\infty} )^{q-1}
\nonumber
\\
&\le 2^{q-1} \norm{\psi_r}_{\infty} \int_{\X} M^{q-1} d\Pi R(t)^{q-1} \left( \norm{\psi_r}_{\text{Lip}(d)}^{q-1} +  \norm{\psi_r}_{\infty}^{q-1} \right)
\nonumber
\\
%&\le 2  \norm{\psi_r}_{\infty} \int_{\X} M d\Pi R(t) (\norm{\phi}_{\text{Lip}(d)} +  \norm{\psi_r}_{\infty})
%\\
&\le 2^{q-1} \int_{\X} M^{q-1} d\Pi R(t)^{q-1} \left( r + r^{q} \right).
\label{eq:ubconvergence}
\end{align}

We take some inspiration from the truncation techniques in \cite{cattiaux_weak_2007}.
To bound the first part of \eqref{eq:separate_parts}, using convexity,
\eq{
\int_{\X} \left| \phi -  \psi_r  + \int \psi_r d\Pi \right|^q d\Pi
&\le 2^{q-1} \int_{\X} \left| \phi -  \psi_r \right|^q d\Pi  + 2^{q-1} \left| \int \psi_r d\Pi \right|^q.
}
Since $\int_{\X} \phi d\Pi = 0$, then
\eq{
\int_{\X} \psi_r d\Pi 
&= \int_{|\phi| \le r} \phi d\Pi + \int_{|\phi| > r} \psi_r d\Pi 
= -\int_{|\phi| > r} \phi d\Pi + \int_{|\phi| > r} \psi_r d\Pi 
\\
&= \int_{|\phi| > r} (\psi_r - \phi) d\Pi.
}
Since $|\phi -  \psi_r| \le (|\phi| - r) I_{ \{ |\phi| > r \} }$, then using Jensen's inequality,
\eq{
\int_{\X} \left| \phi -  \psi_r  + \int \psi_r d\Pi \right|^q d\Pi
%&\le 2 \int_{\X} ( \phi -  \psi_r )^2 d\Pi  + 2 \left( \int \psi_r d\Pi \right)^2
%\\
&\le 2^q \int_{\X} | \phi -  \psi_r |^q d\Pi
\\
&\le 2^q \int_{|\phi| > r} | |\phi| -  r |^q d\Pi.
}
By Hölder's inequality and Markov's inequality, then for $p > 1$,
\begin{align}
\int_{\X} \left| \phi -  \psi_r  + \int \psi_r d\Pi \right|^q d\Pi
%&\le 4 \int_{|\phi| > r} (|\phi| -  r)^2 d\Pi
%\\
&\le 2^q \int_{|\phi| > r} |\phi|^q d\Pi
\nonumber
\\
&\le 2^q \left( \int_{\X} | \phi |^{q p} d\Pi \right)^{1/p} \Pi(|\phi| > r)^{1 - 1/p}
\nonumber
\\
&\le 2^q \frac{ \int_{\X} |\phi|^{q p} d\Pi }{ r^{q (p - 1)} }.
\label{eq:ubtail}
\end{align}

Combining \eqref{eq:ubconvergence} and \eqref{eq:ubtail}, for $r > 0$,
\eq{
\int_{\X} | \P_t \phi |^2 d\Pi
&\le 2^{2q-2} \int_{\X} M^{q-1} d\Pi R(t)^{q-1} (r + r^q)
+ \frac{ 2^{2q - 1} \int_{\X} |\phi|^{q p} d\Pi }{r^{q (p-1)}}.
}
So when $\int_{\X} |\phi|^{q p} d\Pi$ is not small as in many applications, we can expect the first term on the right hand side with $r^2$ to dominate and minimizing this, we get
\[
r = \left( \frac{ 2 (p-1) \int_{\X} |\phi|^{q p} d\Pi }{\int_{\X} M^{q-1} d\Pi R(t)^{q-1}} \right)^{\frac{1}{q p}}
\]
We then get the upper bound
\begin{align}
\int_{\X} | \P_t \phi |^q d\Pi
&\le 2^{2q-2} \left( 2 (p-1 ) \right)^{\frac{1}{qp}} \left[ \int_{\X} M^{q-1} d\Pi \right]^{1-\frac{1}{qp}}  R(t)^{(q-1) (1-\frac{1}{qp})} \left[ \int_{\X} |\phi|^{qp} d\Pi \right]^{\frac{1}{qp}} 
\nonumber
\\
&+ 2^{2q-2} c_p \left[ \int_{\X} M^{q-1} d\Pi \right]^{1-\frac{1}{p}} R(t)^{(q-1) (1-1/p)} \left[ \int_{\X} |\phi|^{q p} d\Pi \right]^{1/p}
\label{eq:ub_messy}
\end{align}
where $c_p =  2^{1/p} \left[ (p-1)^{1/p} + \frac{ 1 }{ (p-1)^{1-1/p} } \right]$.
We can maximize $c_p$ so that $c_p \le 3$ and $( 2(p-1 ) )^{\frac{1}{q p}} \le ( 9/4 )^{1/q}$. 

To prove the first result, since $M \ge 1$, there is a constant $c \le 2^{2q-2} 3$ such that
\eq{
\int_{\X} | \P_t \phi |^q d\Pi
&\le c \left[ \int_{\X} M^{q-1} d\Pi \right]^{1-\frac{1}{qp}} R(t)^{(q-1)(1-1/p)} \left[ \int_{\X} |\phi|^{q p} d\Pi \right]^{1/p}
\\
&\hspace{.3cm}\left\{
R(t)^{(q-1) (\frac{1}{p}-\frac{1}{qp} )} \left[ \int_{\X} |\phi|^{q p} d\Pi \right]^{(\frac{1}{qp}-\frac{1}{p} )}
+ 1
\right\}.
}
If $R(t)^{q-1} \le \int_{\X} |\phi|^{q p} d\Pi$, then the first result follows.

For the second result, since $R \le 1$ and $M \ge 1$, then using \eqref{eq:ub_messy} and Young's inequality
\eq{
&\int_{\X} \left| \P_t \phi - \int_{\X} \phi d\Pi \right|^q d\Pi
\\
&\le c \left[ \int_{\X} M^{q-1} d\Pi \right]^{1-\frac{1}{q p}}  R(t)^{(q-1)(1-\frac{1}{p})} 
\left[
\left[ \int_{\X} |\phi - \int_{\X} \phi d\Pi |^{qp} d\Pi \right]^{\frac{1}{qp}} 
+ \left[ \int_{\X} |\phi - \int_{\X} \phi d\Pi|^{qp} d\Pi \right]^{\frac{1}{p}}
\right]
\\
&\le c \left[ \int_{\X} M^{q-1} d\Pi \right]^{1-\frac{1}{q p}}  R(t)^{(q-1)(1-\frac{1}{p})}  
\left[
\left( 1 + \frac{1}{q} \right) \left[ 
\int_{\X} |\phi - \int_{\X} \phi d\Pi |^{qp} d\Pi \right]^{\frac{1}{p}} 
+ 1 - \frac{1}{q}
\right].
}
By Jensen's inequality,
$
\int_{\X} \left| \phi - \int_{\X} \phi d\Pi \right|^{qp} d\Pi
\le \int_{\X} d(x, y)^{qp} d\Pi(x) d\Pi(y)
$
and the proof is complete using $\phi = (f - \int_{\X} f d\Pi)/ \norm{f}_{Lip(d)}$.
In the reversible case, $R(\cdot)$ can be replaced with $R(2 \cdot)$ in \eqref{eq:ubconvergence} and the proof is similar.
\end{proof}

Theorem~\ref{theorem:nonreversible_bound} allows estimating the mean of unbounded Lipschitz functions that are important such as estimating the Bayesian posterior mean for AI applications.
Once $t$ is large enough based on the difficulty of estimation represented by the $SFR$, then the variance decays by a combination of the rate $R(\cdot)$ and the Lebesgue space norm where larger $p$ gives faster rates but generally larger constants.
When $q = 2$, the upper bound requires the integrability of $M$ where it is not guaranteed that $M$ is integrable, but often in many examples the condition can be shown.
For example, the integrability of $M$ can be established via a drift condition as we explore in applications later and generally holds when Assumption~\ref{assumption:convergence} gives geometric convergence based on drift and local coupling conditions.
Alternatively, the power of the rate $R(\cdot)^{q-1}$ could be replaced by $R(\cdot)^q$ if $\int_{\X} M^q d\Pi < \infty$ and this would also imply Assumption~\ref{assumption:convergence} directly implies control of the variance for bounded and Lipschitz functions, but this requirement can be limiting.

In the reversible case when $q = 2$, the rate $R(\cdot)$ can be replaced with an improvement $R(2 \cdot)$ and will yield equivalent conditions for a weak Poincare inequality \citep[Theorem 2.3]{rockner_weak_2001}.
In particular, Theorem~\ref{theorem:nonreversible_bound} gives sufficient conditions for reversible processes that for an initial probability measure $\nu$ with Lipschitz density $d\nu / d\Pi$, 
\eq{
&\chi^2( \nu \P_t \mid \Pi)
\le 24 \left( \int_{\X} M d\Pi \right)^{1-\frac{1}{2p}} R(2t)^{1-\frac{1}{p}} \norm{\frac{d\nu}{d\Pi} - 1}_{L_{2p}(\Pi)}^2
}
if $t$ is large enough depending on the SFR of the density $d\nu / d\Pi$.
This is a powerful result as weak convergence bounds can be extended to stronger convergence bounds bypassing tools based around local Doeblin conditions.
There are many potential applications for controlling the chi-square divergence to estimate bounds on the density $d\nu\P_t/d\Pi$ for Monte Carlo algorithms.
The Lipschitz assumption on the density has potentially many applications to improve mixing times in MCMC by relaxing some of the stricter ``warm-start" conditions where the density of the initial measure is assumed uniformly bounded.

This upper bound can also lead to lower bounds via the Paley-Zygmund inequality. In particular, for $\e \in (0, 1)$, Theorem~\ref{theorem:nonreversible_bound} allows us in the reversible case if $t$ is large enough that
\begin{align*}
\Pi\left( \P_t \left( \frac{d\nu}{d\Pi} \right) \ge 1 - \e \right)
%&\ge \frac{ \e^2 }{ \norm{ \P_t f - \Pi f }_{L_2(\Pi)}^2 + \e^2 }
%\\
&\ge \frac{
\e^2
}{ 
24 ( \int_{\X} M d\Pi )^{1-1/(2p)} R(2t)^{1-1/p} \norm{\frac{d\nu}{d\Pi} - 1}_{L_{2p}(\Pi)}^2 + \e^2.
}
\end{align*}

The technique in Theorem~\ref{theorem:nonreversible_bound} can be modified to find improved constants, but the requirement on the precise size of $t$ is less intuitive.

\begin{proposition}
\label{proposition:nonreversible_bound_improved}
Assume the transition kernels $(\P_t)_{t \in T}$ satisfy Assumption~\ref{assumption:convergence}.
For every $p > 1$ and for every Lipschitz function $f : \X \to \R$, if $t \in T$ is large enough so that
\[
t \ge R^{-1}\left[ \frac{ 
(p-1) SFR_{2p}(f)^{2p}
}{
\int_{\X} M d\Pi,
} 
\right]
\]
then
\eq{
\norm{ \P_t f - \int_{\X} f d\Pi }_{L_{2}(\Pi)}^2
\le c_p R(t)^{ 1-1/p }  \left( \int_{\X} M d\Pi \right)^{ 1-1/p } 
\norm{ f - \int_{\X} f d\Pi  }_{L_{2p}(\Pi)}^2
}
where $c_p = 8 \left[ (p-1)^{1/p} + \frac{ 1 }{ (p-1)^{1-1/p} } \right]$.
\end{proposition}

\begin{proof}
Using \eqref{eq:ubconvergence} and \eqref{eq:ubtail} and if $r \ge 1$, we have
\begin{align}
\int_{\X} (\P_t \phi)^2 d\Pi
&\le 4 \int_{\X} M d\Pi R(t) r (1 +  r)
+ \frac{ 8 \int_{\X} |\phi|^{2p} d\Pi }{r^{2 (p-1)}}
\nonumber
\\
&\le 8 \int_{\X} M d\Pi R(t) r^2 + \frac{ 8 \int_{\X} |\phi|^{2p} d\Pi }{r^{2 (p-1)}}.
\label{eq:ub_nonoptimal}
\end{align}
We can choose $t$ large enough so that 
\[
r^2 = \left( 
\frac{ 
(p-1) \int_{\X} |\phi|^{2p} d\Pi  
}{
\int_{\X} M d\Pi R(t)
} 
\right)^{\frac{1}{p}}
\ge 1.
\]
Substituting this $r$ value in concludes the result.
\end{proof}

The convergence bounds become more difficult for functions $f$ that have small SFR. 
If instead we have $SFR_{2}(f) \ge 1$ uniformly bounded below, then taking the limit as $p \to \infty$ in Proposition~\ref{proposition:nonreversible_bound_improved} yields
\eq{
\norm{ \P_t f - \int_{\X} f d\Pi }_{L_{2}(\Pi)}^2
\le C R(t) \int_{\X} M d\Pi
\norm{ f - \int_{\X} f d\Pi  }_{\infty}^2
}
holds for some constant $C > 0$.
If the target measure has Gaussian like tails, that is, for some $\lambda \in (0, 1)$ and $x_0 \in \X$, $\int_{\X} \exp( \lambda d(x, x_0)^2 ) \Pi(dx) < \infty$ then we can obtain a stronger result in general.

\begin{theorem}
\label{theorem:nonreversible_bound_gaussian_tail}
Assume the transition kernels $(\P_t)_{t \in T}$ satisfy Assumption~\ref{assumption:convergence}.
For every $\lambda > 0$, there is an explicit constant $C_\lambda > 0$ such that for every Lipschitz function $f : \X \to \R$, 
\begin{align*}
\norm{ \P_t f - \int_{\X}  f d\Pi}_{L_2(\Pi)}^2
&\le
C_{\lambda}  R(t) \int_{\X} M d\Pi \left[ 1 + \log\left( \frac{ \norm{ \exp(\lambda d(\cdot, \cdot)) }_{L_2(\Pi \otimes \Pi)}^2 }{R(t)}\right) \right] \norm{f}_{\text{Lip}(d)}^2.
\end{align*}
\end{theorem}

\begin{proof}
Let $\phi : \X \to \R$ with $\norm{\phi}_{\text{Lip}(d)} = 1$.
By Markov's inequality, for $p > 1$,
\begin{align*}
\left( \int_{\X} |\phi|^{2p} d\Pi \right)^{1/p} \Pi(|\phi| > r)^{1 - 1/p}
&\le  \frac{p}{\lambda^2} \left( \int_{\X} \exp( \lambda |\phi|^{2}) d\Pi \right)^{1/p} \Pi(|\phi| > r)^{1 - 1/p}
\\
&\le \frac{p}{\lambda^2} \int_{\X} \exp( \lambda |\phi|^{2}) d\Pi \exp(-r^2 \lambda (1-1/p)).
\end{align*}
Combining the previous bounds \eqref{eq:ubconvergence} and \eqref{eq:ubtail},
\eq{
&\int_{\X} (\P_t \phi)^2 d\Pi
\le 4 \int_{\X} M d\Pi R(t) (r + r^2)
+ \frac{8 p}{\lambda^2} \exp(-r^2 \lambda (1-1/p)) \int_{\X} \exp( |\phi|^{2}) d\Pi 
}
holds for all $r > 0$.
We can choose $p$ large enough so that $r^2 \ge 0$ with
\[
r^2 = \frac{1}{\lambda( 1-1/p )} \log\left( \frac{ 2 (p - 1) \int_{\X} \exp( \lambda |\phi|^{2}) d\Pi }{ \lambda \int_{\X} M d\Pi R(t) } \right).
\]
%Pluggin this in and using $\sqrt{x} \le x/2 + 1/2$,
Plugging this in, there is a constant $C_\lambda > 0$ where
\eq{
\int_{\X} (\P_t \phi)^2 d\Pi
\le 
C_\lambda  R(t) \int_{\X} M d\Pi\left\{ \log\left( \int_{\X} \exp( \lambda |\phi|^{2}) d\Pi \right)
+ 1
\right\}.
}
\end{proof}

The weak convergence rate in Assumption~\ref{assumption:convergence} characterizes the $L_2(\Pi)$ rate of convergence for unbounded Lipschitz functions even in the non-reversible case as the following equivalence shows.

\begin{proposition}
\label{proposition:equivalence}
Let $(\P_t)_{t \in T}$ be Markov transition kernels with invariant measure $\Pi$.
Assume $\int_{\X} d(x, x_0)^{2p} \Pi(dx) < \infty$ for some $p > 1$ and some $x_0 \in \X$.
Then for a rate function $R : T \to (0, 1]$ strictly decreasing to $0$, the following are equivalent:
\begin{enumerate}
\item 
There are constants $C_1 > 0$ and $\kappa_1 \in (0, 1]$ such that for all $t \in T$
\begin{align}
\norm{ \P_t f - \int_{\X} f d\Pi }_{L_1(\Pi)}
\le C_1 R(t)^{\kappa_1} \left( \norm{f}_{\infty} + \norm{f}_{\text{Lip}(d)} \right)
\label{equiv:weak_convergence}
\end{align}
holds for all bounded Lipschitz functions $f : \X \to \R$.

\item 
There are constants $C_2 > 0$ and $\kappa_2 \in (0, 1]$ such that for all $t \in T$
\begin{align}
\norm{ \P_t f - \int_{\X} f d\Pi }_{L_2(\Pi)}^2
\le C_2 R(t)^{\kappa_2} \norm{f}_{\text{Lip}(d)}^2
\label{equiv:var_convergence}
\end{align}
holds for all Lipschitz functions $f : \X \to \R$.
\end{enumerate}

\end{proposition}

\begin{proof}
For \eqref{equiv:weak_convergence} implies \eqref{equiv:var_convergence}, the proof of Theorem~\ref{theorem:nonreversible_bound} \eqref{eq:ubconvergence} holds under the condition \eqref{equiv:weak_convergence} so that for some constant $C_3 > 0$
\[
\sup_{ \norm{f}_{\text{Lip}(d)} \le 1 } \norm{ \P_t f - \int_{\X} f d\Pi }_{L_2(\Pi)}^2
\le C_3 R(t)^{(1-1/p) \kappa_1}.
\]
The converse follows by Jensen's inequality
\[
\sup_{ \norm{f}_{\text{Lip}(d)} \le 1 } \norm{ \P_t f - \int_{\X} f d\Pi }_{L_1(\Pi)}
\le \sup_{ \norm{f}_{\text{Lip}(d)} \le 1 } \norm{ \P_t f - \int_{\X} f d\Pi }_{L_2(\Pi)}
\le C_2^{1/2} R(t)^{\kappa_2/2}.
\]
\end{proof}

\subsection{Applications to central limit theorems}
\label{section:clt}

Theorem~\ref{theorem:nonreversible_bound} has many applications to controlling the variance of averages and central limit theorems for Lipschitz functions.
Most importantly, this does not require commonly used techniques based on total variation to control the limiting variance for the central limit theorem \cite{jones:2004}.
We focus on the discrete case for simplicity.

\begin{corollary}
\label{corollary:variance_bounds}
Let $(X_t)_{t \in \Z_+}$ be a stationary Markov process with invariant measure $\Pi$ and transition kernels $(\P_t)_{t \in \Z_+}$ satisfying Assumption~\ref{assumption:convergence}. 
For some $p > 1$, define $r_{p, \infty}(m) = \int_m^{\infty} R(u)^{\frac{1}{2} \left( 1-1/p \right)} du$ for every $m \in \Z_+$.
For all centered Lipschitz functions $f : \X \to \R$ with $\int_{\X} f d\Pi = 0$, the following hold:

\begin{enumerate}
\item
If
$
n - 1 \ge t \ge R^{-1}\left[ 
SFR_{2p}(f)^{2p}
\right], then
$ 
\eq{
&\E\left( \left| \frac{1}{n} \sum_{s = 0}^{n-1} f(X_{s}) \right|^2 \right)
\le \norm{ f }_{L_{2p}(\Pi)}^2 \frac{2}{n} \left\{ t
+  2 \sqrt{6} \left[ \int_{\X} M d\Pi \right]^{\frac{1}{2} ( 1-\frac{1}{2p} )} 
\left( R(t)^{\frac{1}{2} \left( 1-\frac{1}{p} \right)} + r_{p, \infty}(m)  \right)
\right\}.
}

\item For every $n \in \Z_+$,
\eq{
\E\left( \left| \frac{1}{n} \sum_{s = 0}^{n-1} f(X_{s}) \right|^2 \right)
\le \frac{1}{n} \norm{ f }_{\text{Lip}(d)}^2 K_{\Pi}
\left[ 1 + 2 C_{2, \Pi}^{1/2} \left( \int_{\X} M d\Pi \right)^{\frac{1}{2} (1 - \frac{1}{2p})} r_{p, \infty}(0) 
\right]
}
where $K_{\Pi} = \norm{d}_{L_2(\Pi \otimes \Pi)} \vee \norm{d}_{L_2(\Pi \otimes \Pi)}^2$ and $C_{2, \Pi}$ is defined in Theorem~\ref{theorem:nonreversible_bound}.

\end{enumerate}
\end{corollary}

\begin{proof}
Let $f : \X \to \R$ with $\phi = f - \int_{\X} f d\Pi$ and write $\bar{\phi} = n^{-1} \sum_{s = 0}^{n-1} \phi(X_{s})$.
Since we have assumed a stationary process, using Hölder's inequality,
\eq{
\E\left(  \left| \bar{\phi} \right|^2 \right)
&\le \frac{1}{n} \norm{ \phi }_{L_2(\Pi)}^2 + \frac{2}{n} \sum_{k = 1}^{n-1} \int_{\X} \left| \P_k \phi \right| \left| \phi \right| d\Pi
\\
&\le \frac{1}{n} \norm{ \phi }_{L_{2}(\Pi)}^2 + \frac{2}{n} \norm{ \phi }_{L_{2}(\Pi)} \sum_{k = 1}^{n-1} \norm{ \P_k \phi }_{L_{2}(\Pi)}.
}
We can use the first result of Theorem~\ref{theorem:nonreversible_bound} with constant $c_1 \le 24$ to get 
\eq{
\E\left(  \left| \bar{\phi} \right|^2 \right)
&\le \frac{1}{n} \norm{ \phi }_{L_{2}(\Pi)}^2 + \frac{2}{n} \norm{ \phi }_{L_{2}(\Pi)} \sum_{k = 1}^{t-1} \norm{ \P_k \phi }_{L_{2}(\Pi)}
+ \frac{2}{n} \norm{ \phi }_{L_{2}(\Pi)} \sum_{k = t}^{n-1} \norm{ \P_k \phi }_{L_{2}(\Pi)}
\\
&\le \frac{2 t}{n} \norm{ \phi }_{L_{2 p}(\Pi)}^2
+ \frac{2}{n} \norm{ \phi }_{L_{2 p}(\Pi)} \sum_{k = t}^{n-1} \norm{ \P_k \phi }_{L_{2}(\Pi)}
\\
&\le \norm{ \phi }_{L_{2p}(\Pi)}^2 \frac{2}{n} \left\{ t
+  c_1^{1/2} \left[ \int_{\X} M d\Pi \right]^{\frac{1}{2} ( 1-\frac{1}{2p} )} 
\left( R(t)^{\frac{1}{2} \left( 1-\frac{1}{p} \right)} + \int_t^{n-1} R(u)^{\frac{1}{2} \left( 1-\frac{1}{p} \right)} du \right)
\right\}.
}
We can use the second result of Theorem~\ref{theorem:nonreversible_bound} with constant $C_{2, \Pi}$ to get
\eq{
\E\left(  \left| \bar{\phi} \right|^2 \right)
&\le \frac{1}{n} \norm{ \phi }_{\text{Lip}(d)}^2 \norm{d}_{L_2(\Pi \otimes \Pi)}^2 + \frac{2}{n} \norm{ \phi }_{\text{Lip}(d)} \norm{d}_{L_2(\Pi \otimes \Pi)} \sum_{k = 1}^{n-1} \norm{ \P_k \phi }_{L_{2}(\Pi)}
\\
&\le \frac{1}{n} \norm{ \phi }_{\text{Lip}(d)}^2 K_{\Pi}
\left[ 1 + 2 C_{2, \Pi}^{1/2} \left( \int_{\X} M d\Pi \right)^{\frac{1}{2} (1 - \frac{1}{2p})} \int_0^{n-1} R(u)^{ \frac{1}{2} \left( 1-\frac{1}{p} \right) } du
\right].
} 
%Here we used the upper bound $\norm{ \phi }_{L_2(\Pi)} \le \norm{ \phi }_{\text{Lip}(d)} \norm{d}_{L_2(\Pi \otimes \Pi)}$.
\end{proof}

When the first bound is finite, then the variance is $L_{2p}(\Pi)$ norm bounded by the function $f$ which is a stronger condition than $L_{2}(\Pi)$ bounded as observed in \cite{roberts_variance_2008} in the reversible geometrically converging case.
The proof for continuous time case is similar or follows directly using the sum
$
\int_0^n f(X_s) ds = \sum_{k = 1}^n \int_{k-1}^{k} f(X_s) ds.
$
Corollary~\ref{corollary:variance_bounds} has direct applications to parallel MCMC simulations with independent initializations and also applications to concentration with arbitrary initial probability measures $\nu$ such as for $\e, q > 0$ 
\eq{
\Prob\left( \left| \frac{1}{n} \sum_{s = 1}^n f(X_{s}) - \int_{\X} f d\Pi \right| \ge \e \right)
&\le \norm{\frac{d\nu}{d\Pi}}_{L_q(\Pi)} \left( \e^{-2} \Var\left( \frac{1}{n}  \sum_{s = 1}^n  f(X_{s}) \right) \right)^{1-1/q}.
}

In particular, the bound on the variance leads to powerful results such as central limit theorems for unbounded Lipschitz functions where Proposition~\ref{corollary:variance_bounds} can be used in combination with a stationary functional central limit theorem \cite{dedecker:rio:2000}.
The following result can be modified for continuous time and importantly, the limiting variance can be upper bounded explicitly.

\begin{corollary}
\label{corollary:clt}
Let $(X_t)_{t \in \Z_+}$ be a stationary Markov process with invariant measure $\Pi$ and transition kernels $(\P_t)_{t \in \Z_+}$ satisfying Assumption~\ref{assumption:convergence}. 
Assume $\int_{\X} M d\Pi < \infty$ and for some $p > 1$
\[
\int_0^{\infty} R(u)^{\frac{1}{2} \left( 1-1/p \right)} du < \infty.
\]
Then for every Lipschitz function $f : \X \to \R$ such that $\norm{f}_{L_{2p}(\Pi)} < \infty$, 
\[
\sigma^2(f) = \lim_{n \to \infty} \Var\left( \frac{1}{\sqrt{n}}  \sum_{s = 0}^{n-1}  f(X_{s}) \right)
\]
is upper bounded and for every $z \in \R$, $\lim_{n \to \infty} \Prob\left( \frac{1}{\sqrt{n}} \sum_{s = 0}^{n-1}  \left( f(X_{s}) - \int_{\X} f d\Pi \right) \le z \right) = \Phi(z)$ where $\Phi$ is a Gaussian distribution function with mean 0 and variance $\sigma^2(f)$.
\end{corollary}

\section{Applications}
\label{section:applications}

\subsection{Applications to the chi-square convergence of infinite-dimensional Metropolis-Hastings}

Markov processes on infinite dimensional spaces can serve as a model for understanding the scaling of high-dimensional Markov process on Euclidean spaces.
Let $\H$ be a separable Hilbert space with norm $\norm{\cdot}$ and $\gamma$ be a Gaussian measure with mean 0 and covariance operator $C$ with $\tr(C) < \infty$.
Let $g : \H \to \R$ and define the target probability measure by its density $d\Pi/d\gamma \propto \exp(-g)$.
When $\H$ is infinite-dimensional, then random-walk Metropolis-Hastings is infeasible where the pCN Metropolis-Hastings algorithm is well-defined.
For $\rho \in [0, 1)$, consider the pCN Metropolis-Hastings algorithm defined by
\[
X_t =
\begin{cases} 
p_t, U_t \le \exp( g(X_{t-1}) - g(p_t) ) \\
X_{t-1}, U_t > \exp( g(X_{t-1}) - g(p_t) )
\end{cases}
\]
where $p_t = \rho X_{t-1} + \sqrt{1 - \rho^2} Z_t$ with independent uniform random variables $(U_t)_{t \in \Z_+}$ and independent $\gamma$ distributed variables $(Z_t)_{t \in \Z_+}$.
We say a function is $\beta$-Hölder continuous for $\beta \in (0, 1]$if it is Lipschitz continuous with respect to $\norm{\cdot}^{\beta}$.
The following shows that convergence in the chi-squared divergence can be obtained from weak convergence even when the convergence is not geometric.
The additional requirement is a Hölder continuous initial measure and a central limit theorem can be shown with Corollary~\ref{corollary:clt} for Hölder continuous functions.

\begin{proposition}
Let $(\P_t)_{t \in \Z_+}$ be the transition kernels for pCN Metropolis-Hastings.
Assume $g$ is $\beta$-Hölder  for $\beta \in (0, 1]$ and $\inf_{x \in \H} g(x) > -\infty$.
Then for any $p > 1$, there exists constants $C, c > 0$ such that
\eq{
\chi^2\left( \nu \P_t \mid \Pi \right)
\le C \exp(- c \log(2t)^{2/\beta} )
\norm{\frac{d\nu}{d\Pi}}_{\text{Lip}(\norm{\cdot}^{\beta})}^2
}
holds for all probability measures $\nu$ on $\H$ with $\beta$-Hölder continuous density $d\nu/d\Pi$.
\end{proposition}

\begin{proof}
For some small $\theta > 0$, $\int \exp(\theta \norm{x}^2) \Pi(dx) < \infty$ by Fernique's theorem.
The weak convergence bound follows from \citep[Proposition 12, Proposition 13, and Lemma 18]{durmus:etal:2016}.
Let $L = \norm{g}_{\text{Lip}(\norm{\cdot}^\beta)}$ and $\kappa = \theta / (36 L^{2/\beta} )$ and $s = (1 - \rho)^2 \theta / 16$.
Then for some $C > 0$, for all $x \in \H$
\[
\W_{\norm{\cdot}^{\beta} \wedge 1}\left( \P_t(x, \cdot), \Pi \right)
\le C \exp(-\kappa {\log(t)}^{2/\beta} ) \exp(s \norm{x}^2).
\]
Since $s \le \theta$ and $\inf_{x \in \H} g(x) > -\infty$, then $\int_{\H} \exp(s \norm{x}^2) d\Pi < \infty$.
Clearly, 
\[
\norm{ \P_t(x, \cdot) - \Pi }_{BL( \norm{\cdot}^{\beta} )}
\le 2 \W_{\norm{\cdot}^{\beta} \wedge 1}\left( \P_t(x, \cdot), \Pi \right).
\]
Then Theorem~\ref{theorem:nonreversible_bound_gaussian_tail} can be applied.
Since Metropolis-Hastings is reversible for $\Pi$, then applying this to $\P_t (d\nu / d\Pi)$ gives the chi-squared divergence bound.
\end{proof}

\subsection{Applications to the stability of stochastic gradient descent}

It is generally easier to show informative convergence results in Wasserstein distances for Stochastic Gradient Descent (SGD) than in total variation.
In particular, the theoretical results developed here can be used to control the variance for SGD often used in AI based upon convergence in Wasserstein distances.
We are interested in applying SGD to functions that fail to be strongly convex and decay less quickly than quadratic towards infinity. 
Let $F: \R^d \to \R$ be a twice continuously differentiable and we will assume the gradient $\nabla F$ is bounded and Lipschitz continuous, and $F$ is locally strongly convex on $\norm{x}_2 \le r$ for any $r > 0$.
By locally strongly convex, there is a constant $m_r > 0$ such that for all $\norm{x}_2 \le r$,  the Hessian $\nabla^2 F(x)$ is positive definite, that is, for all $v \in \R^d$
\[
v^T \nabla^2 F(x) v \ge m_r \norm{v}_2^2.
\]
For example, the function $x \mapsto \sqrt{\norm{x}_2^2 + 1}$ is only locally strongly convex.
We will also assume $F$ satisfies a tail condition that for $c > 0$ and $\alpha \in (0, 1)$,
\begin{align}
\nabla F(x)^T x \ge c \norm{x}_2^{2(1- \alpha)}
\label{assumption:sgd_f}
\end{align}
holds for all $x \in \R^d$.

With step size $h \in (0, 1)$, we can define an SGD algorithm through the updates
\[
X_t = X_{t - 1} - h \nabla F(X_{t-1}) + h \xi_t.
\]
where $(\xi_t)_{t \ge 1}$ are independent random variables on $\R^d$ with $\E(\xi_t) = 0$.
We will assume for every $q \ge 1$,
\begin{align}
\sup_{t \in \Z_+} \E\left( \norm{ \xi_t }_2^{2q} \right) < \infty.
\label{assumption:sgd_noise}
\end{align}

\begin{proposition}
\label{proposition:sgd_convergence}
Let $(\P_t)_{t \in \Z_+}$ be the transition kernels for stochastic gradient descent defined in this section.
If the step size $h$ is chosen small enough, then for every $p \ge 2$ and a constant $K_p > 0$ depending on $p$
\begin{align}
\norm{ \P_t f - \int_{\X}  f d\Pi}_{L_2(\Pi)}^2
&\le \frac{K_p}{( t + 1)^{\frac{p-1}{2\alpha} } } \norm{ f }_{\text{Lip}(\norm{\cdot}_2)}^2
\end{align}
holds for all Lipschitz functions $f : \X \to \R$.
\end{proposition}

The rate is sufficiently fast and the methods developed here can be used to show a central limit theorem rather generally with Corollary~\ref{corollary:clt}.
However, it seems possible to obtain a faster subexponential rate of convergence under additional assumptions.
To prove Proposition~\ref{proposition:sgd_convergence}, we show the local contraction in a Wasserstein distance relying on the local strong convexity of the function $F$.
For $x, y \in \R^d$, define $x' = x - h \nabla F(x) + h \xi$ and $y' = y - h \nabla F(y) + h \xi$ with $\xi$ satisfying \eqref{assumption:sgd_noise}.

\begin{proposition}
\label{proposition:sgd_contraction}
Let $(\P_t)_{t \in \Z_+}$ be the transition kernels for stochastic gradient descent defined in this section.
Let $r > 0$, and let $m_r$ be the local strong convexity constant of $F$ when $\norm{x}_2 \le r$. 
Define $L = \norm{\nabla F}_{\text{Lip}(\norm{\cdot}_2)}$ and if $h \le 2/(L + m_r)$, then if either $\norm{x}_2 \le r$ and $\norm{y}_2 \le r$ or $d_r(x, y) < 1$,
\[
\W_{d_r}( \P_1(x, \cdot), \P_1(y, \cdot) ) \le \left( 1 - \frac{2 h}{\frac{1}{L} + \frac{1}{m_r}} \right)^{1/2} d_r(x, y).
\]
where $d_r(x, y) = \norm{x - y}/(2r) \wedge 1$.
\end{proposition}

\begin{proof}
Then by strong convexity \citep[Theorem 2.1.12]{nesterov_lectures_2018}
\eq{
\norm{x' - y'}_2^2
&= \norm{x - y}_2^2 
+ h^2 L^2 \norm{\nabla F(x) - \nabla F(y)}_2^2 
- 2 h ( \nabla F(x) - \nabla F(y) ) \cdot (x - y)
\\
&\le \left( 1 - \frac{2 h m_r L}{L + m_r} \right) \norm{x - y}_2^2 
+ h \left( h  - \frac{2}{L + m_r} \right) \norm{\nabla F(x) - \nabla F(y)}_2^2 
\\
&\le \left( 1 - \frac{2 h m_r L}{L + m_r} \right) \norm{x - y}_2^2.
}
If $d_r(x, y) \ge 1$, then $\E\left( d_r(x', y') \right) \le d_r(x, y)$ and if $d_r(x, y) \le 1$,
\[
\W_{d_r}( \P_1(x, \cdot), \P_1(y, \cdot) )
\le \E\left( d_r(x', y') \right) 
\le \left( 1 - \frac{2 h m_r L}{L + m_r} \right)^{1/2} d_r(x, y).
\]
If $\norm{x}_2, \norm{y}_2 \le r$, then
\[
\E\left( d_r(x', y') \right) \le \left( 1 - \frac{2 h m_r L}{L + m_r} \right)^{1/2}.
\]
\end{proof}

For $p \ge 2$, define a drift function by $V : \R^d \to [0, \infty)$ by $V(x) = ( 1 + \norm{x}_2^2 )^{p/2}$.
We show a polynomial drift condition relying on the tail condition \eqref{assumption:sgd_f} and the moments on the noise \eqref{assumption:sgd_noise}.

\begin{proposition}
\label{proposition:sgd_drift}
Let $(\P_t)_{t \in \Z_+}$ be the transition kernels for stochastic gradient descent.
Then there are constants $k, K > 0$ such that for all $x \in \R^d$
\[
\P_1 V(x) - V(x)
\le - k  V(x)^{ 1 - \frac{2\alpha}{p} } + K. 
\]
\end{proposition}

\begin{proof}
By Taylor expansion and using \eqref{assumption:sgd_f} and \eqref{assumption:sgd_noise}, there is a constant $C > 0$ such that
\eq{
&\E\left[ V(x') - V(x) \right]
\\
&\le - h \nabla V(x) \cdot \nabla F(x)
+ \E\left( (x' - x)^T \int_0^1 \nabla^2 V( x + s (x' - x) ) (1-s)ds (x' - x) \right)
\\
&\le - h p \left( 1 + \norm{x}_2^2 \right)^{p/2-1} x \cdot \nabla F(x) 
+ C h^2 \left( 1 + \norm{x}_2^2 \right)^{p/2-1}
\\
%&\le - h p \left( 1 + \norm{x}^2 \right)^{p/2-\alpha} 
%+ C h^2 \left( 1 + \norm{x}^2 \right)^{p/2-1}
%\\
&\le - h p V(x)^{1-\frac{2\alpha}{p}} 
+ C h^2 V(x)^{1-\frac{2\alpha}{p}} \left( 1 + \norm{x}_2^2 \right)^{\alpha - 1}.
}
The result follows considering the separate cases when $V(\cdot)$ is small and large.
\end{proof}

\begin{proof}[Proof of Proposition~\ref{proposition:sgd_convergence}] 
Let $p \ge 2$.
The drift shown in Proposition~\ref{proposition:sgd_drift} and the local contraction shown in Proposition~\ref{proposition:sgd_contraction} imply by either techniques in \cite{butkovsky_subgeometric_2014,durmus:etal:2016} that the invariant measure exists and weak a convergence bound holds.
By Proposition~\ref{proposition:sgd_drift} and the assumption on the gradient noise \eqref{assumption:sgd_noise}, then 
\[
\int_{\R^d} ( 1 + \norm{x}_2^2 )^{p/2} d\Pi(x) < \infty.
\]
since this holds for $p = q + 2 \alpha$ with arbitrary $q \ge 2$.
For some $r > 0$ with $d_r(x, y) = \norm{x - y}/(2 r) \wedge 1$ and some $C_p > 0$, for for all $t$ and $x \in \X$, 
\[
\W_{d_r}(\P_t(x, \cdot), \Pi)
\le \frac{C_p}{ (1 + t)^{ \frac{p}{2\alpha} } } \left( ( 1 + \norm{x}_2^2 )^{p/2} + \int_{\R^d} ( 1 + \norm{x}_2^2 )^{p/2} d\Pi(x) \right).
\]
Therefore, we have shown that weak convergence holds in the sense of Assumption~\ref{assumption:convergence} and the proof follows by the second result in Theorem~\ref{theorem:nonreversible_bound}.
\end{proof}

\subsection{Applications to stability of solutions of stochastic delay equations}

We look at convergence of stochastic delay equations that have been previously studied in Wasserstein distances \cite{butkovsky_subgeometric_2014, hairer_asymptotic_2011}.
Let $\C(I, \R^d)$ denote the set of continuous real-valued functions defined on an interval $I \subset \R$.
Let $v : \R \to (0, \infty)$ be a bounded, strictly positive, and strictly monotone increasing function.
Consider the stochastic differential equation
\begin{align}
&dX(t) = -X(t) dt + v(X(t-1)) dW(t) \label{eq:sde}
\\
&X(s) = \eta(s), s \in [-1, 0] 
\end{align}
where the initial condition $\eta \in \C([-1, 0], \R)$ and $W$ is a standard Wiener process on $\R$.

Generally, the solution of such an equation is not a Markov process, but can be lifted to define a Markov process solution on a larger space.
Let $f : \C([-1, 0], \R^d) \to \R^d$ and $g : \C([-1, 0], \R^d) \to \R^d \times \R^d$ and define the stochastic delay equation with $X_t(s) = X_{t + s}$ by
\eq{
&dX(t) = f(X_t) dt + g(X_t) dW(t)
\\ 
&X_0 \in \C([-1, 0], \R^d) 
}
where $W$ is a standard Wiener process on $\R^d$.
The solution $(X_t)_{t \ge 0}$ , if it exists, is a Markov process on $\C$ with transition kernel $\P_t(\eta, \cdot)$ for $\eta \in \C([-1, 0], \R^d)$.

Let $\norm{A}_F = \tr(A^T A)^{1/2}$ denote the Frobenius norm of a matrix $A$.
We will assume $f, g$ are continuous and bounded on bounded subsets on $\C([-1, 0], \R^d)$ and $g^{R}(\eta)$ is the right inverse of the matrix $g(\eta)$ for every $\eta \in \Z$ so that $g(\eta) g^{R}(\eta) = I$ where $I$ is the identity matrix on $\R^d \times \R^d$.
Assume there are constants $K, b, B > 0$ such that for any $x, y \in \C([-1, 0], \R^d)$,
\begin{align}
&2 \ip{f(x) - f(y), x(0) - y(0)} \wedge 0 
+ \norm{ g(x) - g(y) }_{F}^2
\le K \norm{ x - y }^2
\label{assumption:lipschitz}
\\
&\sup_{h \in \C([-1, 0], \R^d)} \norm{g^{R}(h)}_{F} < \infty
\label{assumption:g}
\\
&\ip{f(x), x(0)} \le - b \norm{x(0)}^2 + B.
\label{assumption:drift}
\end{align}
These assumptions ensure there exists a solution to \eqref{eq:sde} and an unique invariant measure.
Solutions to stochastic delay equations are of interest as they often have the property that $\P_t(\eta, \cdot)$ cannot converge in total variation to its invariant measure, and tools for convergence to equilibrium with drift and minorization are unavailable \cite{MT2009}.
However, we are able to convert weak convergence from previous results into providing stronger theoretical results.

\begin{proposition}
Let $(\P_t)_{t \ge 0}$ be the transition kernels for the solution of the stochastic delay equation defined by \eqref{eq:sde}.
Assume the conditions on $f, g$ \eqref{assumption:lipschitz}, \eqref{assumption:g}, and \eqref{assumption:drift} are satisfied.
Then for any $p > 1$, there is an $r > 0$ and a constant $C_p > 0$ such that if if $t \ge -2p \log( SFR_{2p}(f) ) / r$,
\eq{
\norm{ \P_t f - \int_{\C([-1, 0], \R^d)} f d\Pi }_{L_{2}(\Pi)}^2
\le C_p \exp( -(1-1/p) r t ) \norm{f - \int_{\C([-1, 0], \R^d)} f d\Pi}^2_{L_{2p}(\Pi)}
}
holds for all Lipschitz functions $f : \X \to \R$.
\end{proposition}

\begin{proof}
Under these assumptions, we have that there exists a unique global solution over time and there exists a unique invariant measure $\Pi$ \citep[Theorem 3.1]{hairer_asymptotic_2011}.
Let $V(x) = \norm{x(0)}^2$ and \citep[Remark 5.2]{hairer_asymptotic_2011} this is integrable with the invariant measure.
The weak convergence follows by \citep[Proposition 5.3 and Proposition 5.4]{hairer_asymptotic_2011} so that for some $r, \delta > 0$ and for all $\eta \in \C([-1, 0], \R^d)$
\[
\W_{d_\delta}(\P_{t}(\eta, \cdot), \Pi)
\le \exp(-r t) \left( 1 + \sqrt{V(\eta)} + \int_{\C([-1, 0], \R^d)} \sqrt{V} d\Pi \right)
\]
where $d_{\delta}(x, y) = \left( \norm{x - y} / \delta \right) \wedge 1$.
Then we may apply the first part of Theorem~\ref{theorem:nonreversible_bound}.
\end{proof}

\section{Final discussion}
\label{section:conclusion}

The main contributions of this paper develop constructive Lebesgue norm convergence bounds for non-reversible Markov transition kernels based only on weak convergence.
To the best of our knowledge, this is the first set of results extending weak subgeometric convergence rates of the transition kernels $(\P_t)_{t \in T}$ to constructive bounds on the $L_q(\Pi)$ norm of $\P_t f$ for unbounded Lipschitz functions $f$.
The main techniques here appear to generalize to semigroups not necessarily constructed from Markov transition kernels.
Using similar techniques, the relative entropy can be controlled as well and should be of interest to investigate weak convergence and implications related to weak Sobolev inequalities \cite{cattiaux_weak_2007}.

These results have important practical applications to verify the reliability of Markov chain Monte Carlo simulations via central limit theorems when tools to establish geometric ergodicity in total variation are unavailable \cite{butkovsky_subgeometric_2014, durmus:etal:2016}.
We obtain $L_2(\Pi)$ convergence of $\P_t f$, the function $f$ we are estimating is required to be in $L_{2p}(\Pi)$ with $p > 1$ in comparison to the reversible case which requires only $L_{2}(\Pi)$.
In particular, these results are not an improvement in the reversible geometrically converging case.
It would be interesting to further develop a more comprehensive understanding of this requirement as has been done for central limit theorems.
It would also be interesting to see additional equivalent conditions such as developed for total variation and a spectral gap \cite{roberts_variance_2008}.
Important applications for the implications of weak convergence also seem possible to understanding the degeneracy of the total variation distance when analyzing the performance of Markov chain Monte Carlo algorithms due to limitations of random number generation when the state space is not discrete.

\bibliographystyle{imsart-number}
\bibliography{references}

@article{brown:jones:2023,
author = {Austin Brown and Galin Jones},
title = {{Lower bounds on the rate of convergence for accept-reject-based Markov chains in Wasserstein and total variation distances}},
volume = {31},
journal = {Bernoulli},
number = {3},
publisher = {Bernoulli Society for Mathematical Statistics and Probability},
pages = {1908 -- 1928},
year = {2025}
}

@article{komorowski_central_2012,
  title        = {Central limit theorem for {Markov} processes with spectral gap in the {Wasserstein} metric},
  author       = {Komorowski, Tomasz and Walczuk, Anna},
  year         = 2012,
  journal      = {Stochastic Processes and their Applications},
  volume       = 122,
  number       = 5,
  pages        = {2155--2184},
  doi          = {10.1016/j.spa.2012.03.006},
  issn         = {03044149}
}

@book{nesterov_lectures_2018,
  title        = {Lectures on {Convex} {Optimization}},
  author       = {Nesterov, Yurii},
  year         = 2018,
  publisher    = {Springer International Publishing},
  address      = {Cham},
  series       = {Springer {Optimization} and {Its} {Applications}},
  volume       = 137,
  doi          = {10.1007/978-3-319-91578-4},
  isbn         = {978-3-319-91577-7 978-3-319-91578-4}
}

@misc{jin2020wasserstein,
  title        = {Wasserstein Rate Driven CLTs for Markov Chains with Weighted Lipschitz, Sobolev, and Stein Test Functions},
  author       = {Rui Jin and Aixin Tan},
  year         = 2020,
  url          = {https://arxiv.org/abs/2002.09427},
  eprint       = {2002.09427},
  archiveprefix = {arXiv},
  primaryclass = {math.ST}
}

@article{cattiaux_weak_2007,
  title        = {Weak logarithmic {Sobolev} inequalities and entropic convergence},
  author       = {Cattiaux, P. and Gentil, I. and Guillin, A.},
  year         = 2007,
  journal      = {Probability Theory and Related Fields},
  volume       = 139,
  number       = {3-4},
  pages        = {563--603},
  doi          = {10.1007/s00440-007-0054-5},
  issn         = {0178-8051, 1432-2064}
}

@article{roberts_variance_2008,
  title        = {Variance bounding {Markov} chains},
  author       = {Roberts, Gareth O. and Rosenthal, Jeffrey S.},
  year         = 2008,
  journal      = {The Annals of Applied Probability},
  volume       = 18,
  number       = 3,
  doi          = {10.1214/07-AAP486},
  issn         = {1050-5164}
}

@article{jones:2004,
  title        = {{On the Markov chain central limit theorem}},
  author       = {Galin L. Jones},
  year         = 2004,
  journal      = {Probability Surveys},
  publisher    = {Institute of Mathematical Statistics and Bernoulli Society},
  volume       = 1,
  number       = {none},
  pages        = {299 -- 320},
  doi          = {10.1214/154957804100000051}
}

@article{dedecker:rio:2000,
  title        = {On the functional central limit theorem for stationary processes},
  author       = {Jérôme Dedecker and Emmanuel Rio},
  year         = 2000,
  journal      = {Annales de l'Institut Henri Poincare (B) Probability and Statistics},
  volume       = 36,
  number       = 1,
  pages        = {1--34},
  doi          = {https://doi.org/10.1016/S0246-0203(00)00111-4},
  issn         = {0246-0203},
  url          = {https://www.sciencedirect.com/science/article/pii/S0246020300001114}
}

@article{QinHobert2021,
  title        = {On the limitations of single-step drift and minorization in Markov chain convergence analysis},
  author       = {Qin, Qian and Hobert, James P.},
  year         = 2021,
  journal      = {The Annals of Applied Probability},
  volume       = 31,
  number       = 4,
  pages        = {1633--1659},
  doi          = {10.1214/20-AAP1628}
}

@article{kipnis1986central,
  title        = {Central limit theorem for additive functionals of reversible Markov processes and applications to simple exclusions},
  author       = {Kipnis, Claude and Varadhan, SR Srinivasa},
  year         = 1986,
  journal      = {Communications in Mathematical Physics},
  publisher    = {Springer},
  volume       = 104,
  number       = 1,
  pages        = {1--19}
}

@misc{dwivedi_log-concave_2019,
  title        = {Log-concave sampling: {Metropolis}-{Hastings} algorithms are fast},
  author       = {Dwivedi, Raaz and Chen, Yuansi and Wainwright, Martin J. and Yu, Bin},
  year         = 2019,
  publisher    = {arXiv},
  doi          = {10.48550/arXiv.1801.02309},
  note         = {arXiv:1801.02309 [stat]}
}

@article{lovasz_random_1993,
  title        = {Random walks in a convex body and an improved volume algorithm},
  author       = {Lovász, L. and Simonovits, M.},
  year         = 1993,
  journal      = {Random Structures \& Algorithms},
  volume       = 4,
  number       = 4,
  pages        = {359--412},
  doi          = {10.1002/rsa.3240040402},
  issn         = {1042-9832, 1098-2418}
}

@article{qin2021wasserstein,
  title        = {Wasserstein-based methods for convergence complexity analysis of {MCMC} with applications},
  author       = {Qin, Qian and Hobert, James P},
  year         = 2021,
  journal      = {{\rm To appear in} Annals of Applied Probability}
}

@article{Madras2010,
  title        = {{Quantitative bounds for Markov chain convergence: Wasserstein and total variation distances}},
  author       = {Neal Madras and Deniz Sezer},
  year         = 2010,
  journal      = {Bernoulli},
  publisher    = {Bernoulli Society for Mathematical Statistics and Probability},
  volume       = 16,
  number       = 3,
  pages        = {882 -- 908}
}

@book{MT2009,
  title        = {{M}arkov Chains and Stochastic Stability},
  author       = {Meyn, Sean P. and Tweedie, Richard L.},
  year         = 2009,
  publisher    = {Cambridge University Press},
  address      = {USA},
  edition      = 2
}

@article{hairer_asymptotic_2011,
  title        = {Asymptotic coupling and a general form of {Harris}’ theorem with applications to stochastic delay equations},
  author       = {Hairer, M. and Mattingly, J. C. and Scheutzow, M.},
  year         = 2011,
  journal      = {Probability Theory and Related Fields},
  volume       = 149,
  number       = {1-2},
  pages        = {223--259},
  doi          = {10.1007/s00440-009-0250-6},
  issn         = {0178-8051, 1432-2064}
}

@article{bakry_rate_2008,
  title        = {Rate of convergence for ergodic continuous {Markov} processes: {Lyapunov} versus {Poincaré}},
  author       = {Bakry, Dominique and Cattiaux, Patrick and Guillin, Arnaud},
  year         = 2008,
  journal      = {Journal of Functional Analysis},
  volume       = 254,
  number       = 3,
  pages        = {727--759},
  doi          = {10.1016/j.jfa.2007.11.002},
  issn         = {00221236}
}

@article{rockner_weak_2001,
  title        = {Weak {Poincaré} {Inequalities} and {L2}-{Convergence} {Rates} of {Markov} {Semigroups}},
  author       = {Röckner, Michael and Wang, Feng-Yu},
  year         = 2001,
  journal      = {Journal of Functional Analysis},
  volume       = 185,
  number       = 2,
  pages        = {564--603},
  doi          = {10.1006/jfan.2001.3776},
  issn         = {00221236}
}

@article{durmus:etal:2016,
  title        = {Subgeometric rates of convergence in {W}asserstein distance for {M}arkov chains},
  author       = {Alain Durmus and Gersende Fort and {\'E}ric Moulines},
  year         = 2016,
  journal      = {Annales de l'Institut Henri Poincaré, Probabilités et Statistiques},
  publisher    = {Institut Henri Poincaré},
  volume       = 52,
  number       = 4,
  pages        = {1799--1822}
}

@article{douc_practical_2004,
  title        = {Practical drift conditions for subgeometric rates of convergence},
  author       = {Douc, Randal and Fort, Gersende and Moulines, Eric and Soulier, Philippe},
  year         = 2004,
  journal      = {The Annals of Applied Probability},
  volume       = 14,
  number       = 3,
  doi          = {10.1214/105051604000000323},
  issn         = {1050-5164}
}

@article{butkovsky_subgeometric_2014,
  title        = {Subgeometric rates of convergence of {Markov} processes in the {Wasserstein} metric},
  author       = {Butkovsky, Oleg},
  year         = 2014,
  journal      = {The Annals of Applied Probability},
  volume       = 24,
  number       = 2,
  doi          = {10.1214/13-AAP922},
  issn         = {1050-5164}
}

@article{tweedie_modes_1977,
  title        = {Modes of convergence of {Markov} chain transition probabilities},
  author       = {Tweedie, Richard L},
  year         = 1977,
  journal      = {Journal of Mathematical Analysis and Applications},
  volume       = 60,
  number       = 1,
  pages        = {280--291},
  doi          = {10.1016/0022-247X(77)90067-1},
  issn         = {0022247X}
}

@article{hairer_spectral_2014,
  title        = {Spectral gaps for a {Metropolis}-{Hastings} algorithm in infinite dimensions},
  author       = {Hairer, Martin and Stuart, Andrew M. and Vollmer, Sebastian J.},
  year         = 2014,
  journal      = {The Annals of Applied Probability},
  volume       = 24,
  number       = 6,
  doi          = {10.1214/13-AAP982},
  issn         = {1050-5164}
}

\end{document}